\documentclass[oneside]{amsart}
\usepackage[latin9]{inputenc}
\usepackage{geometry}
\usepackage{amstext}
\usepackage{amsthm}
\usepackage{amssymb}
\usepackage{amsrefs}
\usepackage{mathtools}

\makeatletter
\numberwithin{equation}{section}
\numberwithin{figure}{section}
\theoremstyle{plain}
\newtheorem{thm}{\protect\theoremname}[section]
\theoremstyle{remark}
\newtheorem{rem}[thm]{\protect\remarkname}
\theoremstyle{plain}

\theoremstyle{plain}
\newtheorem{lem}[thm]{\protect\lemmaname}
\theoremstyle{plain}
\newtheorem{prop}[thm]{\protect\propositionname}
\theoremstyle{definition}
\newtheorem{defn}[thm]{Definition}
\newtheorem{ex}[thm]{Example}

\DeclareFontEncoding{OT2}{}{}
\DeclareFontSubstitution{OT2}{cmr}{m}{l}

\DeclareFontFamily{OT2}{cmr}{\hyphenchar\font45 }
\DeclareFontShape{OT2}{cmr}{m}{l}{%
<5><6><7><8><9>gen*wncyr%
<10><10.95><12><14.4><17.28><20.74><24.88>wncyr10}{}

\DeclareMathAlphabet{\mathcyr}{OT2}{cmr}{m}{l}
\DeclareMathAlphabet{\mathcyb}{OT2}{cmr}{b}{l}
\SetMathAlphabet{\mathcyr}{bold}{OT2}{cmr}{b}{l}

\allowdisplaybreaks[1]

\usepackage{amstext}
\usepackage{amsthm}
\usepackage{amssymb}
\usepackage{ytableau}

\usepackage{shuffle}
\usepackage[all]{xy}

\newcommand{\sh}{\mathbin{\mathcyr{sh}}}

\newcommand{\tast}{\mathbin{\tilde{\ast}}}

\newcommand{\HH}{\mathfrak{H}}
\newcommand{\RS}{\mathrm{RS}}
\newcommand{\ZZ}{\mathcal{Z}}
\newcommand{\Q}{\mathbb{Q}}
\newcommand{\R}{\mathbb{R}}

\newcommand{\C}{\mathbb{C}}
\DeclareMathOperator{\Rea}{Re}

\DeclareMathOperator{\id}{id}
\DeclareMathOperator{\spa}{span}

\makeatother

\providecommand{\corollaryname}{Corollary}
\providecommand{\lemmaname}{Lemma}
\providecommand{\propositionname}{Proposition}
\providecommand{\remarkname}{Remark}
\providecommand{\theoremname}{Theorem}

\AtBeginDocument{%
   \def\MR#1{}
}

\begin{document}

\title{Ohno relation for regularized refined symmetric multiple zeta values}

\author{Minoru Hirose}
\address[Minoru Hirose]{Graduate School of Science and Engineering, Kagoshima University, 1-21-35 Korimoto, Kagoshima, Kagoshima 890-0065, Japan}
\email{hirose@sci.kagoshima-u.ac.jp}

\author{Hideki Murahara}
\address[Hideki Murahara]{Department of Economics, The University of Kitakyushu, 4-2-1 Kitagata, Kokuraminami-ku, Kitakyushu, Fukuoka, 802-8577, Japan}
\email{hmurahara@mathformula.page}

\author{Shingo Saito}
\address[Shingo Saito]{Faculty of Arts and Science, Kyushu University, 744, Motooka, Nishi-ku, Fukuoka, 819-0395, Japan}
\email{ssaito@artsci.kyushu-u.ac.jp}



\maketitle

\begin{abstract}
The Ohno relation is one of the most celebrated results in the theory of multiple zeta values,
which are iterated integrals from $0$ to $1$.
In a previous paper, the authors generalized the Ohno relation to regularized multiple zeta values, which are non-admissible iterated integrals from $0$ to $1$.
Meanwhile, Takeyama proved an analogue of the Ohno relation for refined symmetric multiple zeta values, which are iterated integrals from $0$ to $0$.
In this paper, we generalize Takeyama's result to regularized refined symmetric multiple zeta values, which are non-admissible iterated integrals from $0$ to $0$.
\end{abstract}

\section{Introduction}
\subsection{Multiple zeta values and the Ohno relation}
By an \emph{index}, we mean a finite (possibly empty) sequence of positive integers.
We say that an index is \emph{admissible} if either it is empty or its last component is greater than $1$.

\begin{defn}[Multiple zeta values]\label{defn:mzv}
 Let $\boldsymbol{k}=(k_1,\dots,k_r)$ be an admissible index.
 We define the \emph{multiple zeta value} as
 \[
  \zeta(\boldsymbol{k})=\sum_{0<n_1<\dots<n_r}\frac{1}{n_1^{k_1}\dotsm n_r^{k_r}}.
 \]
\end{defn}

We understand that this definition entails $\zeta(\emptyset)=1$, where $\emptyset$ denotes the empty index.

Let $\ZZ$ denote the $\Q$-linear subspace of $\R$ spanned by the multiple zeta values.
It has been known that $\ZZ$ is algebraically rich in the sense that the multiple zeta values satisfy a large number of linear and algebraic relations.
The sum formula and the duality relation are among the most significant relations,
and their simultaneous generalization was obtained by Ohno~\cite{Ohno1999}.
To state the Ohno relation, we need to define the dual index of each admissible index.
\begin{defn}[Dual index]\label{defn:dual_index}
Let $\boldsymbol{k}$ be an admissible index.
We can write $\boldsymbol{k}$ uniquely as
\[
 \boldsymbol{k}=(\{1\}^{a_1-1},b_1+1,\dots,\{1\}^{a_h-1},b_h+1),
\]
where $h\ge0$, $a_1,\dots,a_h,b_1,\dots,b_h\ge1$, and $\{1\}^{a}=\underbrace{1,\dots,1}_a$,
and define its \emph{dual} index as
\[
 \boldsymbol{k}^{\dagger}=(\{1\}^{b_h-1},a_h+1,\dots,\{1\}^{b_1-1},a_1+1).
\]
\end{defn}

\begin{ex}
 If $\boldsymbol{k}=(1,1,2,3,1,2)=(\{1\}^{3-1},1+1,\{1\}^{1-1},2+1,\{1\}^{2-1},1+1)$,
 then we have $\boldsymbol{k}^{\dagger}=(\{1\}^{1-1},2+1,\{1\}^{2-1},1+1,\{1\}^{1-1},3+1)=(3,1,2,4)$.
\end{ex}

\begin{thm}[Ohno relation; Ohno~\cite{Ohno1999}]\label{thm:ohno}
 Let $\boldsymbol{k}=(k_1,\dots,k_r)$ be an admissible index and
 write its dual index as $\boldsymbol{k}^{\dagger}=(l_1,\dots,l_s)$.
 Then for every nonnegative integer $m$, we have
 \[
 \sum_{\substack{e_1+\dots+e_r=m\\e_1,\dots,e_r\ge0}}\zeta(k_1+e_1,\dots,k_r+e_r)
  =\sum_{\substack{f_1+\dots+f_s=m\\f_1,\dots,f_s\ge0}}\zeta(l_1+f_1,\dots,l_s+f_s).
 \]
\end{thm}

Theorem~\ref{thm:ohno} with $\boldsymbol{k}=(1,\dots,1,2)$ represents the sum formula,
whereas Theorem~\ref{thm:ohno} with $m=0$ represents the duality relation.

\subsection{Refined symmetric multiple zeta values}
Using the shuffle regularized multiple zeta values
$\zeta^{\sh}(\boldsymbol{k})\in\ZZ$
(see \cite{IharaKanekoZagier2006} for details),
we define the \emph{refined symmetric multiple zeta values} as follows:
\begin{defn}[Refined symmetric multiple zeta values; Bachmann-Tasaka-Takeyama~\cite{BTTCompositio}, {Hirose~\cite[Proposition~4]{Hirose2020}}, Jarossay~\cite{Jaro2Ver4}]
Let $\boldsymbol{k}=(k_1,\dots,k_r)$ be an index.
We define the \emph{refined symmetric multiple zeta values} as
\[
 \zeta_{\RS}(k_1,\dots,k_r)
 =\sum_{\substack{0\le p\le q\le r\\k_{p+1}=\cdots=k_q=1}}
 \frac{(-2\pi i)^{q-p}}{(q-p+1)!}(-1)^{k_{q+1}+\cdots+k_r}
 \zeta^{\sh}(k_1,\dots,k_p)\zeta^{\sh}(k_r,\dots,k_{q+1}).
\]
\end{defn}
The refined symmetric multiple zeta values are a lift of
the \emph{symmetric multple zeta values} $\zeta_S(\boldsymbol{k})\in\ZZ/\pi^2\ZZ$
defined by Kaneko and Zagier~\cite{KanekoZagier2024}:
\begin{align*}
 \zeta_S(\boldsymbol{k})
 &=\Rea(\zeta_{\RS}(\boldsymbol{k}))\bmod\pi^2\ZZ\\
 &=\sum_{\substack{0\le p\le r}}
 (-1)^{k_{p+1}+\cdots+k_r}\zeta^{\sh}(k_1,\dots,k_p)\zeta^{\sh}(k_r,\dots,k_{p+1})\bmod\pi^2\ZZ.
\end{align*}

An analogue of the Ohno relation for refined symmetric multiple zeta values was proved by Takeyama~\cite[Proof of Corollary~3]{Takeyama2020}.
To state the theorem of Takeyama, we need to define the Hoffman dual index, which resembles the dual index described in  Definition~\ref{defn:dual_index}:
\begin{defn}[Hoffman dual index]
 Let $\boldsymbol{k}$ be a nonempty index.
 Then we define its \emph{Hoffman dual index} $\boldsymbol{k}^{\vee}$ as the index obtained by
 writing the components of $\boldsymbol{k}$ as sums of ones and swapping the plus signs and commas.
\end{defn}

\begin{ex}
 If $\boldsymbol{k}=(1,2,2)=(1,1+1,1+1)$, then we have $\boldsymbol{k}^{\vee}=(1+1,1+1,1)=(2,2,1)$.
\end{ex}

\begin{thm}[Takeyama~\cite{Takeyama2020}]\label{thm:takeyama}
 Let $\boldsymbol{k}=(k_1,\dots,k_r)$ be a nonempty index and
 write its Hoffman dual index as $\boldsymbol{k}^{\vee}=(l_1,\dots,l_s)$.
 Then for every nonnegative integer $m$, we have
 \[
  \sum_{\substack{f_1+\dots+f_s=m\\f_1,\dots,f_s\ge0}}\zeta_{\RS}((l_1+f_1,\dots,l_s+f_s)^{\vee})
  =\sum_{j=0}^{m}\frac{(-2\pi i)^{j}}{(j+1)!}
  \sum_{\substack{e_1+\cdots+e_r=m-j\\e_1,\dots,e_r\ge0}}\zeta_{\RS}(k_1+e_1,\dots,k_r+e_r).
 \]
\end{thm}

\begin{rem}
 The $(j+1)!$ in the denominator is missing in Takeyama~\cite{Takeyama2020}.
\end{rem}

Note that this theorem is a lift of the (weak) Ohno relation of symmetric multiple zeta values obtained by Oyama~\cite{Oyama2018}:
\begin{thm}[Oyama~\cite{Oyama2018}]\label{thm:oyama}
 Let $\boldsymbol{k}=(k_1,\dots,k_r)$ be a nonempty index and write its Hoffman dual index as $\boldsymbol{k}^{\vee}=(l_1,\dots,l_s)$.
 Then for every nonnegative integer $m$, we have
 \[
  \sum_{\substack{f_1+\dots+f_s=m\\f_1,\dots,f_s\ge0}}\zeta_S((l_1+f_1,\dots,l_s+f_s)^{\vee})
  =\sum_{\substack{e_1+\cdots+e_r=m\\e_1,\dots,e_r\ge0}}\zeta_S(k_1+e_1,\dots,k_r+e_r).
 \]
\end{thm}

\subsection{Ohno type relations for regularized multiple zeta values}
We let $\mathfrak{h}=\Q\langle x,y\rangle$ denote the $\Q$-algebra of noncommutative polynomials of rational coefficients in two variables $x$ and $y$,
and look at the correspondence between an admissible index $(k_1,\dots,k_r)$ and
$yx^{k_1-1}\dotsm yx^{k_r-1}\in\mathfrak{h}$;
note that
\[
 \spa_{\Q}\{yx^{k_1-1}\dotsm yx^{k_r-1}\mid\text{$(k_1,\dots,k_r)$ is an admissible index}\}
 =\Q+y\mathfrak{h}x.
\]

Define the \emph{shuffle product} $\sh\colon\mathfrak{h}\times\mathfrak{h}\to\mathfrak{h}$ as the $\mathbb{Q}$-bilinear map satisfying
\begin{gather*}
 w\sh 1=1\sh w=w,\\
 u_1w_1\sh u_2w_2=u_1(w_1\sh u_2w_2)+u_2(u_1w_1\sh w_2)
\end{gather*}
for all $u_1,u_2\in\{x,y\}$ and $w,w_1,w_2\in\mathfrak{h}$.
It turns out that there exists a unique $\Q$-linear map $Z^{\sh}\colon\mathfrak{h}\to\ZZ$
satisfying the following conditions:
\begin{enumerate}
  \item $Z^{\sh}(yx^{k_1-1}\dotsm yx^{k_r-1})=\zeta^{\sh}(k_1,\dots,k_r)$ whenever $(k_1,\dots,k_r)$ is an index;
  \item $Z^{\sh}(x)=0$;
  \item $Z^{\sh}(w\sh w')=Z^{\sh}(w)Z^{\sh}(w')$ whenever $w,w'\in\mathfrak{h}$.
\end{enumerate}

In order to describe the Ohno relation in the framework of $\mathfrak{h}$,
we define the homomorphism $\sigma\colon\mathfrak{h}\to\mathfrak{h}[[T]]$ by
$\sigma(x)=x$ and $\sigma(y)=y(1-xT)^{-1}$.
Then
\[
 Z^{\sh}(\sigma(yx^{k_1-1}\dotsm yx^{k_r-1}))
 =\sum_{m=0}^{\infty}\Biggl(\sum_{\substack{e_1+\dots+e_r=m\\e_1,\dots,e_r\ge0}}
  \zeta^{\sh}(k_1+e_1,\dots, k_r+e_r)\Biggr)T^m
\]
for all indices $(k_1,\dots,k_r)$,
where we extend $Z^{\sh}$ coefficientwise
to $Z^{\sh}\colon\mathfrak{h}[[T]]\to\mathcal{Z}[[T]]$;
we will tacitly perform such extensions in the rest of the paper.

We also define the anti-automorphism $\tau\colon\mathfrak{h}\to\mathfrak{h}$ by
$\tau(x)=y$ and $\tau(y)=x$.
Then for an admissible index $(k_1,\dots,k_r)$ and its dual index $(l_1,\dots,l_s)$, we have
\[
 \tau(yx^{k_1-1}\dotsm yx^{k_r-1})=yx^{l_1-1}\dotsm yx^{l_s-1}.
\]
Note that the duality relation implies that $Z^{\sh}(\tau(w))=Z^{\sh}(w)$ for all $w\in\mathfrak{h}$.

We can now restate the Ohno relation (Theorem~\ref{thm:ohno}) as follows:
\begin{thm}
 For $w\in\Q+y\mathfrak{h}x$, we have
 \[
  Z^{\sh}(\sigma(\tau(w)))=Z^{\sh}(\sigma(w)).
 \]
\end{thm}
It is natural to ask whether the above theorem holds for general $w\in\mathfrak{h}$ or equivalently \[
Z^{\sh}\biggl(\sigma\biggl(\tau\biggl(\frac{1}{1-xA}w\frac{1}{1-yB}\biggr)\biggr)\biggr)
\stackrel{?}{=}
  Z^{\sh}\biggl(\sigma\biggl(\frac{1}{1-xA}w\frac{1}{1-yB}\biggr)\biggr)
\]
for $w\in\Q+y\mathfrak{h}x$.
The authors answered the question in the negative by proving  the following result:
\begin{thm}[{\cite[Theorem~5]{HiroseMuraharaSaito2023}}]\label{thm:HMS2023}
 If $w\in\Q+y\mathfrak{h}x$, then we have
 \[  Z^{\sh}\biggl(\sigma\biggl(\tau\biggl(\frac{1}{1-xA}w\frac{1}{1-yB}\biggr)\biggr)\biggr)
  =\frac{\Gamma(1+A)\Gamma(1-T+B)}{\Gamma(1+B)\Gamma(1-T+A)}
Z^{\sh}\biggl(\sigma\biggl(\frac{1}{1-xA}w\frac{1}{1-yB}\biggr)\biggr).
 \]
\end{thm}

If we set $\rho=\tau\circ\sigma\circ\tau$, then Theorem~\ref{thm:HMS2023} and the duality relation show that
\[
 Z^{\sh}\biggl(\rho\biggl(\frac{1}{1-xA}w\frac{1}{1-yB}\biggr)\biggr)
 =\frac{\Gamma(1+A)\Gamma(1-T+B)}{\Gamma(1+B)\Gamma(1-T+A)}
 Z^{\sh}\biggl(\sigma\biggl(\frac{1}{1-xA}w\frac{1}{1-yB}\biggr)\biggr),
\]
whose analogue to the refined symmetric multiple zeta values is our main theorem (Theorem~\ref{thm:main}).

\subsection{Main theorem}
Our main theorem generalizes Theorems~\ref{thm:takeyama}
in the same manner as Theorem~\ref{thm:HMS2023} generalizes the Ohno relation (Theorem~\ref{thm:ohno}).
Just as we extended $\zeta^{\sh}$ to $\mathfrak{h}$,
we need to extend $\zeta_{\RS}$ to $\mathfrak{h}$:
\begin{defn}\label{def:Z_RS}
We define a $\Q$-linear map $Z_{\RS}\colon\mathfrak{h}\to\C$ by
\[
Z_{\RS}(u_1\cdots u_k) =
\sum_{\substack{0\leq p\leq q\leq k\\u_{p+1}=\cdots=u_q=y}}
\frac{(-2\pi i)^{q-p-1}}{(q-p)!}
(-1)^{k-q}
Z^{\sh}(u_1\cdots u_p)
Z^{\sh}(u_k\cdots u_{q+1}),
\]
where $k\ge0$ and $u_1,\dots,u_k\in\{x,y\}$.
\end{defn}

A more conceptual description of $Z_{\RS}$ in terms of non-admissible iterated integrals will be given in Section~\ref{sec:RRSMZV}.
We call the values $Z_{\RS}(u_1\cdots u_k)$ the \emph{regularized refined symmetric multiple zeta values}.
Note that this $Z_{\RS}$ is essentially equal to $Z^{\RS}$
defined in \cite[Section~3.1]{Hirose2020}:
\[
Z_{\RS}(x^{a_1}y\cdots x^{a_{n-1}}yx^{a_n} )
=(-1)^nZ^{\RS}(e_0^{a_1}e_1\cdots e_0^{a_{n-1}}e_1 e_0^{a_n}).
\]
Since
\[
 Z_{\RS}(yx^{k_1-1}\dotsm yx^{k_r-1}y)=\zeta_{\RS}(k_1,\dots,k_r)
\]
for all indices $(k_1,\dots,k_r)$,
the map $Z_{\RS}$ can be thought of as an extension of $\zeta_{\RS}$
from
\[
 \mathfrak{h}^0\coloneqq
 y\Q+y\mathfrak{h}y=\{yx^{k_1-1}\dotsm yx^{k_r-1}y\mid\text{$(k_1,\dots,k_r)$ is an index}\}
\]
 to $\mathfrak{h}$.

In order to work on refined symmetric multiple zeta values,
we define $\widetilde{\sigma},\widetilde{\rho}\colon\mathfrak{h}\to\mathfrak{h}[[T]]$,
the symmetric analogues of $\sigma$ and $\rho$, by
$\widetilde{\sigma}(w)=\sigma(w)(1-xT)$
and
$\widetilde{\rho}(w)=\rho(w)(1-yT)^{-1}$.
Then
\begin{align*}
 Z_{\RS}(\widetilde{\sigma}(yx^{k_1-1}\dotsm yx^{k_r-1}y))
 &=Z_{\RS}\biggl(y\frac{x^{k_1-1}}{1-xT}\dotsm y\frac{x^{k_r-1}}{1-xT}y\biggr)\\
 &=\sum_{m=0}^{\infty}\Biggl(\sum_{\substack{e_1+\dots+e_r=m\\e_1,\dots,e_r\ge0}}
  \zeta_{\RS}(k_1+e_1,\dots k_r+e_r)\Biggr)T^m
\end{align*}
for all indices $(k_1,\dots,k_r)$.

\begin{thm}[Main theorem]\label{thm:main}
 If $w\in\mathfrak{h}^0$, then we have
 \begin{align*}
  &Z_{\RS}\biggl(\widetilde{\rho}\biggl(\frac{1}{1-xA}w\frac{1}{1-xB}\biggr)\biggr)\\
  &=
  \frac{1-e^{-2\pi iT}}{2\pi iT}
   \left(2-\frac{\Gamma(1-T)\Gamma(1+A)}{\Gamma(1-T+A)}\right)
   \left(2-\frac{\Gamma(1+T)\Gamma(1-B)}{\Gamma(1+T-B)}\right)
  Z_{\RS}\biggl(\widetilde{\sigma}\biggl(\frac{1}{1-xA}w\frac{1}{1-xB}\biggr)\biggr).
 \end{align*}
\end{thm}

\begin{rem}
 Setting $A=B=0$ in Theorem~\ref{thm:main} yields Theorem~\ref{thm:takeyama}.
\end{rem}

\section{Regularized refined symmetric multiple zeta values}
\label{sec:RRSMZV}
In this section, we will explain why $Z_{\RS}$ is a natural extension of $\zeta_{\RS}$,
by showing that $Z_{\RS}$ preserves the following properties of $\zeta_{\RS}$:
having an iterated integral expression and
satisfying the duality relation and the symmetric harmonic relation.

\subsection{Iterated integrals, shuffle regularized multiple zeta values, and regularized refined symmetric multiple zeta values}
A \emph{tangential base point} is a pair of $p\in\{0,1\}$ and $v\in\C^{\times}$,
which we will write as $p_{\vec{v}}$.
A path from a tangential base point $p_{\vec{v}}$ to a tangential base point $q_{\vec{w}}$ means a continuous and piecewise smooth map $\gamma\colon[0,1]\to \C$ such that $\gamma((0,1))\subset \C\setminus\{0,1\}$, $\gamma(0)=p$, $\gamma(1)=q$, $\gamma'(0)=v$ and $\gamma'(1)=-w$.
Now, let $\gamma$ be a path from $p_{\vec{v}}$ to $q_{\vec{w}}$, and $\omega_1,\dots,\omega_k\in\C \frac{dt}{t} + \C\frac{dt}{1-t}$ be differential forms. Then there exist unique complex numbers $c_j=c_j(\gamma;\omega_1,\dots,\omega_k)\ (j=0,\dots,k)$ such that 
\[
\int_{\varepsilon<t_1<\cdots<t_k<1-\varepsilon}\omega_1(\gamma(t_1))\cdots\omega_k(\gamma(t_k)) = \sum_{j=0}^{k}c_j(\log\varepsilon)^{j} + O(\varepsilon(\log \varepsilon)^k)
\]
for $\varepsilon \to +0$, and we put
\[
I_\gamma(\omega_1,\dots,\omega_k) = c_0(\gamma;\omega_1,\dots,\omega_k).
\]
Then the maps $Z^{\sh}$ and $Z_{\RS}$ can be characterized by
\[
Z^{\sh}(u_1\cdots u_k)=I_{\mathrm{dch}}(\omega_{u_1}\cdots \omega_{u_k} ),\qquad
Z_{\RS}(u_1\cdots u_k) = -\frac{1}{2\pi i}I_{\beta}(\omega_{u_1},\cdots,\omega_{u_k} )
\]
for $u_1,\dots,u_k\in \{x,y\}$, where
\[
\omega_x(t) = \frac{dt}{t},\quad  \omega_y(t) = \frac{dt}{1-t}
\]
and $\mathrm{dch}$ and $\beta$ are a straight path from $0':=0_{\vec{1}}$ to $1':=1_{\overrightarrow{-1}}$
and a simple path from $0'$ to $0'$ which encircles $1$ counterclockwise once,
respectively (see \cite{Hirose2020} for details).

\subsection{Duality relation for regularized refined symmetric multiple zeta values}
Set $z=x+y\in\mathfrak{h}$.
We define an algebra automorphism $\varphi\colon\mathfrak{h}\to\mathfrak{h}$ by setting
$\varphi(x)=z$ and $\varphi(y)=-y$.
Note that $\varphi\circ \varphi=\id$.
The following theorem is called the duality relation for refined symmetric multiple zeta values:
\begin{thm}[{\cite[Theorem~10]{Hirose2020}}]\label{thm:duality_RS}
 If $w\in\mathfrak{h}^0$, then
 \[
  Z_{\RS}(\varphi(w))=-\overline{Z_{\RS}(w)}.
 \]
\end{thm}

We now generalize Theorem~\ref{thm:duality_RS} to regularized refined symmetric multiple zeta values.
\begin{prop}\label{prop:phi-RS}
 If $w\in\mathfrak{h}^0$, then we have
 \[
  Z_{\RS}\biggl(\varphi\biggl(\frac{1}{1-xA}w\frac{1}{1-xB}\biggr)\biggr)
  =-e^{-\pi i(A+B)}\overline{Z_{\RS}\biggl(\frac{1}{1-xA}w\frac{1}{1-xB}\biggr)}.
 \]
\end{prop}

\begin{proof}
 It suffices to show that
 \[
  Z_{\RS}(\varphi(x^{a}u_1\cdots u_kx^{b}))  =-\sum_{a'=0}^{a}\sum_{b'=0}^{b}\frac{(-\pi i)^{a'+b'}}{a'!b'!}\overline{Z_{\RS}(x^{a-a'}u_1\cdots u_kx^{b-b'})}
 \]
 for all $a,b\ge0$, $k\ge1$, and all $u_1,\dots,u_k\in\{x,y\}$ with $u_1=u_k=y$.

 Put $0''=0_{\overrightarrow{-1}}$.
 Let $\beta'$ be the path from $0''$ to $0''$
 obtained by applying the M\"{o}bius transformation $t\mapsto t/(t-1)$ to $\beta$.
 Extend $\omega_x(t)=dt/t$ and $\omega_y(t)=dt/(1-t)$
 to $\omega_{ax+by}(t)=a\,dt/t+b\,dt/(1-t)$ for $a,b\in\Q$.
 Then we have
 \begin{align*}
  Z_{\RS}(\varphi(x^{a}u_1\cdots u_kx^{b}))
  &=-\frac{1}{2\pi i}I_{\beta}(\underbrace{\omega_{\varphi(x)},\dots,\omega_{\varphi(x)}}_{a},\omega_{\varphi(u_1)},\dots,\omega_{\varphi(u_k)},\underbrace{\omega_{\varphi(x)},\dots,\omega_{\varphi(x)}}_{b})\\
  &=-\frac{1}{2\pi i}I_{\beta'}(\underbrace{\omega_x,\dots,\omega_x}_{a},\omega_{u_1},\dots,\omega_{u_k},\underbrace{\omega_x,\dots,\omega_x}_{b})
 \end{align*}
 Now, since $\beta'$ is the composite of the path $\alpha_1$ going clockwise from $0_{\overrightarrow{-1}}$ to $0_{\vec{1}}$,
 the complex conjugate $\overline{\beta}$ of the path $\beta$,
 and the path $\alpha_2$ going counterclockwise from $0_{\vec{1}}$ to $0_{\overrightarrow{-1}}$,
 the path-composition formula implies that
 \begin{align*}
  &I_{\beta'}(\underbrace{\omega_x,\dots,\omega_x}_{a},\omega_{u_1},\dots,\omega_{u_k},\underbrace{\omega_x,\dots,\omega_x}_{b})\\
  &=\sum_{a'=0}^{a}\sum_{b'=0}^{b}
  I_{\alpha_1}(\underbrace{\omega_x,\dots,\omega_x}_{a'})
  I_{\overline{\beta}}(\underbrace{\omega_x,\dots,\omega_x}_{a-a'},\omega_{u_1},\dots,\omega_{u_k},\underbrace{\omega_x,\dots,\omega_x}_{b-b'})
  I_{\alpha_1}(\underbrace{\omega_x,\dots,\omega_x}_{b'})\\
  &=\sum_{a'=0}^{a}\sum_{b'=0}^{b}\frac{(-\pi i)^{a'}}{a'!}
  \overline{I_{\beta}(\underbrace{\omega_x,\dots,\omega_x}_{a-a'},\omega_{u_1},\dots,\omega_{u_k},\underbrace{\omega_x,\dots,\omega_x}_{b-b'})}
  \frac{(-\pi i)^{b'}}{b'!}.
 \end{align*}
 It follows that
 \[
  Z_{\RS}(\varphi(x^{a}u_1\cdots u_kx^{b}))
  =-\sum_{a'=0}^{a}\sum_{b'=0}^{b}
  \frac{(-\pi i)^{a'+b'}}{a'!b'!}\overline{Z_{\RS}(x^{a-a'}u_1\cdots u_kx^{b-b'})},
 \]
 as required.
\end{proof}

\subsection{Symmetric harmonic relation for regularized refined symmetric multiple zeta values}
For an index $\boldsymbol{k}=(k_1,\dots,k_r)$,
we write $w_{\boldsymbol{k}}=yx^{k_1-1}\dotsm yx^{k_r-1}y$.
Define the \emph{symmetric harmonic product} $\tast\colon\mathfrak{h}^0\times\mathfrak{h}^0\to\mathfrak{h}^0$
as the $\mathbb{Q}$-bilinear map satisfying
\[
 w_{\boldsymbol{k}}\tast w_{\boldsymbol{l}}=w_{\boldsymbol{k}*\boldsymbol{l}}
\]
for all indices $\boldsymbol{k}$ and $\boldsymbol{l}$,
where $*$ denotes the harmonic product of indices.
For example,
\[
 yxy\tast yxy=w_{(2)}\tast w_{(2)}=w_{(2)*(2)}=2w_{(2,2)}+w_{(4)}=2yxyxy+yx^3y.
\]
We extend $\tast$ to the $\mathbb{Q}$-bilinear map
$\tast\colon\mathfrak{h}y\mathfrak{h}\times\mathfrak{h}y\mathfrak{h}\to\mathfrak{h}y\mathfrak{h}$ by setting
\[
 x^{a_1}w_1x^{b_1}\tast x^{a_2}w_2x^{b_2}
 =x^{a_1+a_2}(w_1\tast w_2)x^{b_1+b_2}
\]
for all $a_1,b_1,a_2,b_2\ge0$ and $w_1,w_2\in\mathfrak{h}^0$.

\begin{thm}[{\cite[Theorem~8]{Hirose2020}}]\label{thm:harmonic_RS}
 If $w_1,w_2\in\mathfrak{h}^0$, then we have
 \[
  Z_{\RS}(w_1\tast w_2)=Z_{\RS}(w_1)Z_{\RS}(w_2).
 \]
\end{thm}

The following theorem, established by Hirose and Kawamura~\cite[Remark~5.2]{HiroseKawamura2023},
generalizes Theorem~\ref{thm:harmonic_RS} to regularized refined symmetric multiple zeta values:

\begin{thm}[{\cite[Remark~5.2]{HiroseKawamura2023}}]\label{thm:ast_hom_extension}
 If $w_1,w_2\in\mathfrak{h}^0[[A,B]]$, then we have
 \[
  Z_{\RS}\biggl(\frac{1}{1-xA}(w_1\tast w_2)\frac{1}{1-xB}\biggr)
  =Z_{\RS}\biggl(\frac{1}{1-xA}w_1\frac{1}{1-xB}\biggr)Z_{\RS}\biggl(\frac{1}{1-xA}w_2\frac{1}{1-xB}\biggr).
 \]
\end{thm}

Since our notation is rather different from that in \cite{HiroseKawamura2023},
we include its proof for the reader's convenience.

\begin{lem}\label{lem:Z_RS_1/(1-xA)}
 If $(k_1,\dots,k_r)$ is an index, then we have
 \begin{align*}
  &Z_{\RS}\biggl(\frac{1}{1-xA}yx^{k_1-1}\dots yx^{k_r-1}y\frac{1}{1-xB}\biggr)\\
  &=
  \sum_{j=0}^{r}(-1)^{k_{j+1}+\cdots+k_{r}}\zeta_{\mathrm{shift}}^{A,*}\biggl(k_{1},\dots,k_{j};-\frac{\pi i}{2}\biggr)\zeta_{\mathrm{shift}}^{-B,*}\biggl(k_{r},\dots,k_{j+1};\frac{\pi i}{2}\biggr),
 \end{align*}
 where
 \[
 \zeta_{\mathrm{shift}}^{A,*}(k_{1},\dots,k_{r};T)
=\sum_{\substack{a_1+\dots+a_r=a\\a_1,\dots,a_r\ge0}}(-A)^{a_1+\dots+a_r}\zeta^*(k_1+a_1,\dots,k_r+a_r;T)\prod_{j=1}^{r}\binom{k_j-1+a_j}{a_j}
\]
and $\zeta^*$ denotes the harmonic regularized multiple zeta values (see \cite{IharaKanekoZagier2006} for details).
\end{lem}

\begin{proof}
It suffices to show that
\[
  Z_{\RS}\left(x^{a}yx^{k_{1}-1}\cdots yx^{k_{r}-1}yx^{b}\right)
  =\sum_{j=0}^{r}(-1)^{k_{j+1}+\cdots+k_{r}+b}
  \zeta_{a}^{\ast}\biggl(k_{1},\dots,k_{j};-\frac{\pi i}{2}\biggr)\zeta_{b}^{\ast}\biggl(k_{r},\dots,k_{j+1};\frac{\pi i}{2}\biggr)
\]
for all nonnegative integers $a$ and $b$, where
\[
\zeta_{a}^{\bullet}(k_{1},\dots,k_{r};T)=(-1)^{a}\sum_{\substack{a_{1}+\cdots+a_{r}=a\\
a_{1},\dots,a_{r}\geq0
}
}\zeta^{\bullet}(k_{1}+a_{1},\dots,k_{r}+a_{r};T)\prod_{j=1}^{r}\binom{k_{j}-1+a_{j}}{a_{j}}
\]
for $\bullet\in\{*,\sh\}$,
so that
 \[
 \zeta_{\mathrm{shift}}^{A,*}(k_{1},\dots,k_{r};T)
=\sum_{a=0}^{\infty}A^{a}\zeta_{a}^{*}(k_{1},\dots,k_{r};T).
\]

Define an $\mathbb{R}$-linear morphism $\rho:\mathbb{R}[T]\to\mathbb{R}[T]$
by
\[
\rho(e^{Tu})=A(u)e^{Tu},
\]
where 
\[
 A(u) 
 =\exp\left(\sum_{n=2}^{\infty}\frac{(-1)^{n}}{n}\zeta(n)u^{n}\right).
\]
We first note that
\begin{align*}
 & \sum_{l_{1},l_{2}\geq 0}\frac{(-1)^{l_{2}}}{l_{1}!l_{2}!}\rho^{-1}(T^{l_{1}})\vert_{T=-\pi i/2}\cdot\rho^{-1}(T^{l_{2}})\vert_{T=\pi i/2}\cdot u^{l_{1}+l_{2}}\\
 & =\rho^{-1}(\exp(uT))\vert_{T=-\pi i/2}\cdot\rho^{-1}(\exp(-uT))\vert_{T=\pi i/2}\\
 & =A(u)^{-1}e^{-\frac{\pi i}{2}u}\cdot A(-u)^{-1}e^{-\frac{\pi i}{2}u}
 = \frac{\sin(\pi u)}{\pi u} e^{-\pi i u}
 = \frac{1-e^{-2\pi i u}}{2\pi i u}.
\end{align*}
Thus,
\[
\sum_{l_1+l_2=s}\frac{(-1)^{l_{2}}}{l_{1}!l_{2}!}\rho^{-1}(T^{l_{1}})\vert_{T=-\pi i/2}\cdot\rho^{-1}(T^{l_{2}})\vert_{T=\pi i/2} = \frac{(-2\pi i)^{s}}{(s+1)!}.
\]
It follows from the regularization theorem (\cite{IharaKanekoZagier2006}) that
\begin{align*}
 & \sum_{j=0}^{r}(-1)^{k_{j+1}+\cdots+k_{r}+b}\zeta_{a}^{*}\biggl(k_{1},\dots,k_{j};-\frac{\pi i}{2}\biggr)\zeta_{b}^{*}\biggl(k_{r},\dots,k_{j+1};\frac{\pi i}{2}\biggr)\\
 & =\sum_{j=0}^{r}(-1)^{k_{j+1}+\cdots+k_{r}+b}\zeta_{a}^{*}(k_{1},\dots,k_{j};T)\vert_{T=-\pi i/2}\cdot\zeta_{b}^{*}(k_{r},\dots,k_{j+1};T)\vert_{T=\pi i/2}\\
 & =\sum_{j=0}^{r}(-1)^{k_{j+1}+\cdots+k_{r}+b}\rho^{-1}(\zeta_{a}^{\sh}(k_{1},\dots,k_{j};T))\vert_{T=-\pi i/2}\cdot\rho^{-1}(\zeta_{b}^{\sh}(k_{r},\dots,k_{j+1};T))\vert_{T=\pi i/2}\\
 & =\sum_{\substack{0\leq p\leq j\leq q\leq r\\
k_{p+1}=\cdots=k_{j}=1\\
k_{j+1}=\cdots=k_{q}=1
}
}(-1)^{k_{j+1}+\cdots+k_{r}+b}\zeta_{a}^{\sh}(k_{1},\dots,k_{p})\zeta_{b}^{\sh}(k_{r},\dots,k_{q+1})\\
&\hphantom{ =\sum_{\substack{0\leq p\leq j\leq q\leq r\\
k_{p+1}=\cdots=k_{j}=1\\
k_{j+1}=\cdots=k_{q}=1
}
}}\times\frac{1}{(j-p)!(q-j)!}\rho^{-1}(T^{j-p})\vert_{T=-\pi i/2}\cdot\rho^{-1}(T^{q-j})\vert_{T=\pi i/2}\\
 & =\sum_{\substack{0\leq p\leq q\leq r\\
k_{p+1}=\cdots=k_{q}=1
}
}(-1)^{k_{q+1}+\cdots+k_{r}+b}\zeta_{a}^{\sh}(k_{1},\dots,k_{p})\zeta_{b}^{\sh}(k_{r},\dots,k_{q+1})\\
&\hphantom{ =\sum_{\substack{0\leq p\leq q\leq r\\
k_{p+1}=\cdots=k_{q}=1
}
}}\times\sum_{j=p}^{q}\frac{(-1)^{q-j}}{(j-p)!(q-j)!}\rho^{-1}(T^{j-p})\vert_{T=-\pi i/2}\cdot\rho^{-1}(T^{q-j})\vert_{T=\pi i/2}\\
 & =\sum_{\substack{0\leq p\leq q\leq r\\
k_{p+1}=\cdots=k_{q}=1}}
(-1)^{k_{q+1}+\cdots+k_{r}+b}\zeta_{a}^{\sh}(k_{1},\dots,k_{p})\zeta_{b}^{\sh}(k_{r},\dots,k_{q+1})\frac{(-2\pi i)^{q-p}}{(q-p+1)!}\\
 &=Z_{\RS}
  \left(x^{a}yx^{k_{1}-1}\cdots yx^{k_{r}-1}yx^{b}\right).\qedhere
\end{align*}
\end{proof}

\begin{proof}[Proof of Theorem~\ref{thm:ast_hom_extension}]
 Immediate from Lemma~\ref{lem:Z_RS_1/(1-xA)} by observing that
 $\zeta_{\mathrm{shift}}^{A,*}$ satisfies the harmonic relation
 (see \cite[Corollary~2.5]{HiroseKawamura2023}, for example).
\end{proof}

\section{Proof of our main theorem}
In this section, we give a proof of our main theorem (Theorem~\ref{thm:main}).

\begin{prop}[{\cite[Proposition~7]{HiroseMuraharaSaito2023}}]
\label{prop:tau_o_tau}
For $w\in\mathfrak{h}'$, we have
\[
 \rho(w)
 =\sigma(w)+\varphi
 \left(
  \frac{yT}{1+yT}
  z\tast
  \varphi(\sigma(w))\right).
\]
\end{prop}

\begin{rem}
 The phrase ``For $w\in\mathfrak{h}$'' in \cite{HiroseMuraharaSaito2023} is incorrect because
 $(yT/(1+yT))z\tast\varphi(\sigma(1))=(yT/(1+yT))z\tast1$ is not defined.
\end{rem}

\begin{lem}\label{lem:o_tilde}
For $w\in\mathfrak{h}y\mathfrak{h}$, we have
\[
 \widetilde{\rho}(w)
  =\varphi
   \biggl(
    \frac{-y}{1+yT}
    \tast\varphi
  (
   \widetilde{\sigma}(w)
  )
  \biggr).
\]
\end{lem}

\begin{proof}
Since $\sigma(w(1-xT))=\widetilde{\sigma}(w)$ for all $w\in\mathfrak{h}y\mathfrak{h}$,
Proposition~\ref{prop:tau_o_tau} with $w$ replaced by $w(1-xT)$ gives
\[
 \rho(w(1-xT))=\widetilde{\sigma}(w)+\varphi\biggl(\frac{yT}{1+yT}z\tast\varphi(\widetilde{\sigma}(w))\biggr).
\]
Since
\begin{align*}
 \varphi\biggl(\frac{yT}{1+yT}z\tast\varphi(\widetilde{\sigma}(w))\biggr)
 &=\varphi\biggl(\frac{yT}{1+yT}x\tast\varphi(\widetilde{\sigma}(w))\biggr)
   +\varphi\biggl(\frac{yT}{1+yT}y\tast\varphi(\widetilde{\sigma}(w))\biggr)\\
 &=\varphi\biggl(\frac{y}{1+yT}\tast\varphi(\widetilde{\sigma}(w))\biggr)zT
   +\varphi\biggl(\biggl(\frac{-y}{1+yT}+y\biggr)\tast\varphi(\widetilde{\sigma}(w))\biggr)\\
 &=-\varphi\biggl(\frac{-y}{1+yT}\tast\varphi(\widetilde{\sigma}(w))\biggr)zT
   +\varphi\biggl(\frac{-y}{1+yT}\tast\varphi(\widetilde{\sigma}(w))\biggr)
   -\varphi(\varphi(\widetilde{\sigma}(w)))\\
 &=\varphi\biggl(\frac{-y}{1+yT}\tast\varphi(\widetilde{\sigma}(w))\biggr)(1-zT)-\widetilde{\sigma}(w),
\end{align*}
it follows that
\[
 \rho(w(1-xT))=\varphi\biggl(\frac{-y}{1+yT}\tast\varphi(\widetilde{\sigma}(w))\biggr)(1-zT).
\]
Therefore the lemma follows from $\rho(w(1-xT))=\widetilde{\rho}(w)(1-zT)$.
\end{proof}

\begin{lem}\label{lem:ZRS-tilde-rho}
 If $w\in\HH^0$, then we have
 \[
  Z_{\RS}\biggl(\widetilde{\rho}\biggl(\frac{1}{1-xA}w\frac{1}{1-xB}\biggr)\biggr)
  =\overline{Z_{\RS}\biggl(\frac{1}{1-xA}\frac{-y}{1+yT}\frac{1}{1-xB}\biggr)}
  Z_{\RS}\biggl(\widetilde{\sigma}\biggl(\frac{1}{1-xA}w\frac{1}{1-xB}\biggr)\biggr).
 \]
\end{lem}

\begin{proof}
 Lemma~\ref{lem:o_tilde} implies that the left-hand side is equal to 
 \[
  Z_{\RS}\biggl(\varphi\biggl(\frac{-y}{1+yT}\tast\varphi\biggl(\widetilde{\sigma}\biggl(\frac{1}{1-xA}w\frac{1}{1-xB}\biggr)\biggr)\biggr)\biggr).
 \]
 Set
 \begin{align*}
  w'&=(1-xA)\varphi\biggl(\widetilde{\sigma}\biggl(\frac{1}{1-xA}w\frac{1}{1-xB}\biggr)\biggr)(1-xB)\\
  &=(1-xA)\frac{1}{1-(x+y)A}\varphi(\widetilde{\sigma}(w))\frac{1}{1-(x+y)B}(1-xB)\\
  &=\biggl(1+yA\frac{1}{1-(x+y)A}\biggr)\varphi(\widetilde{\sigma}(w))\biggl(1+\frac{1}{1-(x+y)B}yB\biggr)\\
  &\in\HH^0[[A,B]].
 \end{align*}
 Then Proposition~\ref{prop:phi-RS} and Theorem~\ref{thm:ast_hom_extension}
 show that the left-hand side of the required identity is equal to
 \begin{align*}
  &Z_{\RS}\biggl(\varphi\biggl(\frac{-y}{1+yT}\tast\frac{1}{1-xA}w'\frac{1}{1-xB}\biggr)\biggr)\\
  &=Z_{\RS}\biggl(\varphi\biggl(\frac{1}{1-xA}\biggl(\frac{-y}{1+yT}\tast w'\biggr)\frac{1}{1-xB}\biggr)\biggr)\\
  &=-e^{-\pi i(A+B)}
  \overline{Z_{\RS}\biggl(\frac{1}{1-xA}\biggl(\frac{-y}{1+yT}\tast w'\biggr)\frac{1}{1-xB}\biggr)}\\
  &=-e^{-\pi i(A+B)}
  \overline{Z_{\RS}\biggl(\frac{1}{1-xA}\frac{-y}{1+yT}\frac{1}{1-xB}\biggr)
   Z_{\RS}\biggl(\frac{1}{1-xA}w'\frac{1}{1-xB}\biggr)},
 \end{align*}
 and we have
 \begin{align*}
  -e^{-\pi i(A+B)}\overline{Z_{\RS}\biggl(\frac{1}{1-xA}w'\frac{1}{1-xB}\biggr)}
  &=-e^{-\pi i(A+B)}\overline{Z_{\RS}\biggl(\varphi\biggl(\widetilde{\sigma}\biggl(\frac{1}{1-xA}w\frac{1}{1-xB}\biggr)\biggr)\biggr)}\\
  &=Z_{\RS}\biggl(\widetilde{\sigma}\biggl(\frac{1}{1-xA}w\frac{1}{1-xB}\biggr)\biggr).
 \end{align*}
 Hence the proof of the lemma is complete.
\end{proof}

\begin{lem}\label{lem:computation}
We have
\[
Z_{\RS}\left(\frac{1}{1-xA}\frac{-y}{1+yT}\frac{1}{1-xB}\right)
=\frac{e^{2\pi iT}-1}{2\pi iT}
   \left(2-\frac{\Gamma(1-T)\Gamma(1+A)}{\Gamma(1-T+A)}\right)
   \left(2-\frac{\Gamma(1+T)\Gamma(1-B)}{\Gamma(1+T-B)}\right).
\]
\end{lem}

\begin{proof}
Define a $\Q$-linear map $Y:\mathfrak{h}\to \C$ by $Y(y^n)=(-2\pi i)^{n}/n!$ and $Y(\mathfrak{h}x\mathfrak{h})=\{0\}$. Then for $k>0$, Definition \ref{def:Z_RS} can be written as
\begin{align*}
Z_{\mathrm{RS}}(u_{1}\cdots u_{k}) & =-\frac{1}{2\pi i}\sum_{0\leq p\leq q\leq k}(-1)^{k-q}Z^{\sh}(u_{1}\cdots u_{p})Y(u_{p+1}\cdots u_{q})Z^{\sh}(u_{k}\cdots u_{q+1})\\
 & =-\frac{1}{2\pi i}\sum_{0\leq p<q\leq k}(-1)^{k-q}Z^{\sh}(u_{1}\cdots u_{p})Y(u_{p+1}\cdots u_{q})Z^{\sh}(u_{k}\cdots u_{q+1}).
\end{align*}
because
\[
\sum_{0\leq p\leq k}(-1)^{k-p}u_{1}\cdots u_{p}\sh u_{k}\cdots u_{p+1}=0.
\]
It follows that
\[
 Z_{\RS}\left(\frac{1}{1-xA}\frac{-y}{1+yT}\frac{1}{1-xB}\right)
 =-\frac{1}{2\pi i}
  Z^{\sh}\left(\frac{1}{1-xA}\frac{1}{1+yT}\right)
  Y\left(\frac{-y}{1+yT}\right)
  Z^{\sh}\left(\frac{1}{1+xB}\frac{1}{1-yT}\right).
\]
By using the duality relation and the formulas
\[
  Z^{\sh}\left(\frac{1}{1-yC}\frac{1}{1-xD}\right)
  =2-\frac{\Gamma(1-C)\Gamma(1-D)}{\Gamma(1-C-D)},
\]
we obtain
\begin{align*}
 Z_{\RS}\left(\frac{1}{1-xA}\frac{-y}{1+yT}\frac{1}{1-xB}\right)
 &=-\frac{1}{2\pi i}
  Z^{\sh}\left(\frac{1}{1-yT}\frac{1}{1+xA}\right)
  Y\left(\frac{-y}{1+yT}\right)
  Z^{\sh}\left(\frac{1}{1+yT}\frac{1}{1-xB}\right)\\
 &=\frac{1}{2\pi i}
   \left(2-\frac{\Gamma(1-T)\Gamma(1+A)}{\Gamma(1-T+A)}\right)
   \frac{e^{2\pi iT}-1}{T}
   \left(2-\frac{\Gamma(1+T)\Gamma(1-B)}{\Gamma(1+T-B)}\right),
\end{align*}
as required.
\end{proof}

\begin{proof}[Proof of Theorem~\ref{thm:main}]
 Immediate from Lemmas~\ref{lem:ZRS-tilde-rho} and \ref{lem:computation}.
\end{proof}

\section*{Acknowledgements}
This work was supported by JSPS KAKENHI Grant Numbers JP18K03243, JP18K13392, JP22K13897, JP22K03244, and JP23K03072.

\bibliographystyle{amsplain}
\bibliography{References} 

\begin{bibdiv}
\begin{biblist}

\bib{BTTCompositio}{article}{
      author={Bachmann, Henrik},
      author={Takeyama, Yoshihiro},
      author={Tasaka, Koji},
       title={Cyclotomic analogues of finite multiple zeta values},
        date={2018},
        ISSN={0010-437X},
     journal={Compos. Math.},
      volume={154},
      number={12},
       pages={2701\ndash 2721},
      review={\MR{3873529}},
}

\bib{Hirose2020}{article}{
      author={Hirose, Minoru},
       title={Double shuffle relations for refined symmetric multiple zeta values},
        date={2020},
        ISSN={1431-0635,1431-0643},
     journal={Doc. Math.},
      volume={25},
       pages={365\ndash 380},
      review={\MR{4106896}},
}

\bib{HiroseKawamura2023}{misc}{
      author={Hirose, Minoru},
      author={Kawamura, Hanamichi},
       title={On a lifting of $t$-adic symmetric multiple zeta values},
        date={2023},
        note={\texttt{arXiv:2311.00473 [math.NT]}},
}

\bib{HiroseMuraharaSaito2023}{article}{
      author={Hirose, Minoru},
      author={Murahara, Hideki},
      author={Saito, Shingo},
       title={Ohno relation for regularized multiple zeta values},
        date={2023},
        ISSN={0025-5645,1881-1167},
     journal={J. Math. Soc. Japan},
      volume={75},
      number={4},
       pages={1177\ndash 1193},
         url={https://doi.org/10.2969/jmsj/89088908},
      review={\MR{4659311}},
}

\bib{IharaKanekoZagier2006}{article}{
      author={Ihara, Kentaro},
      author={Kaneko, Masanobu},
      author={Zagier, Don},
       title={Derivation and double shuffle relations for multiple zeta values},
        date={2006},
        ISSN={0010-437X,1570-5846},
     journal={Compos. Math.},
      volume={142},
      number={2},
       pages={307\ndash 338},
         url={https://doi.org/10.1112/S0010437X0500182X},
      review={\MR{2218898}},
}

\bib{Jaro2Ver4}{misc}{
      author={Jarossay, David},
       title={Adjoint cyclotomic multiple zeta values and cyclotomic multiple harmonic values},
        date={2019},
        note={\texttt{arXiv:1412.5099 [math.NT]}},
}

\bib{KanekoZagier2024}{unpublished}{
      author={Kaneko, Masanobu},
      author={Zagier, Don},
       title={Finite multiple zeta values},
        note={in preparation},
}

\bib{Ohno1999}{article}{
      author={Ohno, Yasuo},
       title={A generalization of the duality and sum formulas on the multiple zeta values},
        date={1999},
        ISSN={0022-314X,1096-1658},
     journal={J. Number Theory},
      volume={74},
      number={1},
       pages={39\ndash 43},
         url={https://doi.org/10.1006/jnth.1998.2314},
      review={\MR{1670544}},
}

\bib{Oyama2018}{article}{
      author={Oyama, Kojiro},
       title={Ohno-type relation for finite multiple zeta values},
        date={2018},
        ISSN={1340-6116,1883-2032},
     journal={Kyushu J. Math.},
      volume={72},
      number={2},
       pages={277\ndash 285},
         url={https://doi.org/10.2206/kyushujm.72.277},
      review={\MR{3884532}},
}

\bib{Takeyama2020}{article}{
      author={Takeyama, Yoshihiro},
       title={Derivations on the algebra of multiple harmonic {$q$}-series and their applications},
        date={2020},
        ISSN={1382-4090,1572-9303},
     journal={Ramanujan J.},
      volume={52},
      number={1},
       pages={41\ndash 65},
         url={https://doi.org/10.1007/s11139-019-00139-y},
      review={\MR{4083223}},
}

\end{biblist}
\end{bibdiv}
\end{document}